\documentclass[11pt,oneside,reqno]{amsart}

\usepackage{amsmath}
\usepackage{amssymb}
\newtheorem{theorem}{Theorem}[section]

\newtheorem{corollary}{Corollary}[section]
\newtheorem{lemma}{Lemma}[section]

\numberwithin{equation}{section}

\def \N{\mathbb{N}}

\def \Q{\mathbb{Q}}

\def \Z{\mathbb{Z}}

\def\house#1{\setbox1=\hbox{$\,#1\,$}%
\dimen1=\ht1 \advance\dimen1 by 2pt \dimen2=\dp1 \advance\dimen2 by 2pt
\setbox1=\hbox{\vrule height \dimen1 depth\dimen2\box1\vrule}%
\setbox1=\vbox{\hrule\box1}%
\advance\dimen1 by .4pt \ht1=\dimen1
\advance\dimen2 by .4pt \dp1=\dimen2 \box1\relax}
\begin{document}

\title{On  simultaneous approximation of algebraic numbers}
\author{ Veekesh Kumar and R. Thangadurai}%
\address[Veekesh Kumar]{Institute of Mathematical Sciences, A CI of Homi Bhabha National Institute, C.I.T Campus, Taramani, Chennai 600 113, India.}
\address[]{(Current Affiliation): National Institute of Science Education and Research Bhubaneswar, Khurda 752050, Odisha, India.}

\address[R. Thangadurai]{Harish-Chandra Research Institute, A CI of Homi Bhabha National Institute\\
Chhatnag Road, Jhunsi\\
Prayagraj (Allahabad) - 211019\\
INDIA.}

\email[Veekesh Kumar]{veekeshiitg@gmail.com}
\email[R. Thangadurai]{thanga@hri.res.in}
\bigskip
\begin{abstract} 
Let $\Gamma\subset \overline{\mathbb Q}^{\times}$ be a finitely generated multiplicative group of algebraic numbers.  Let $\alpha_1,\ldots,\alpha_r\in\overline{\mathbb Q}^\times$  be  algebraic numbers which are $\mathbb{Q}$-linearly independent  and let $\epsilon>0$  be a given real number.  One of the main results that we prove in this article is as follows; There exist only finitely many tuples $(u, q,  p_1,\ldots,p_r)\in\Gamma\times\mathbb{Z}^{r+1}$ with $d = [\mathbb{Q}(u):\mathbb{Q}]$ for some integer $d\geq 1$ satisfying $|\alpha_i q u|>1$, $\alpha_i q u$  is not a pseudo-Pisot number for some integer $i\in\{1, \ldots, r\}$ and
$$
0<|\alpha_j qu-p_j|<\frac{1}{H^\epsilon(u)|q|^{\frac{d}{r}+\varepsilon}}
$$  
for all integers $j = 1, 2,\ldots, r$, where $H(u)$ is  the absolute Weil height.  In particular, when  $r =1$, this result was proved by Corvaja and Zannier in \cite{corv}.  As an application of our result, we also prove a transcendence criterion which generalizes a result of  Han\v{c}l,  Kolouch, Pulcerov\'a and \v{S}t\v{e}pni\v{c}ka in \cite{hancl}. The proofs  rely on the clever use of  the subspace theorem and the underlying ideas from the work of  Corvaja and Zannier. 
\end{abstract}
\maketitle

\section{introduction}

Rational approximation is a fascinating  and one of the important techniques to prove transcendental results. This area has very rich history and one of the major milestones   is a famous result of Roth \cite{roth1} extending the earlier works of Thue and Siegel (for a proof, see \cite{satha}). Then  Ridout \cite{ridd1} proved a $p$-adic version of Roth, a vast generalization is due to Schmidt \cite{schmidt} (Subspace Theorem) and many versions of the subspace theorem are available now, and these versions are applied to many problems in various branches of Number Theory (for instance, see \cite{zannier}).     
In 2004, Corvaja and Zannier \cite{corv} proved a `Thue-Roth' type inequality with `moving targets' to solve a problem of Mahler.  Recently in 2019, Kulkarni, Mavraki and Nguyen \cite{kul} proved a generalization of the Mahler problem part of  Corvaja and Zannier.  In this article, we are interested in the simultaneous approximation of algebraic numbers in the same spirit of Corvaja and Zannier. Also, we apply our main theorem  to prove a transcendental result. In order to state our results, we shall start with some terminology. 

\bigskip
In order to state our results, we start with the following definition. An algebraic number $\alpha$ is said to be a {\it pseudo-Pisot number}, if $|\alpha|>1$, $\alpha$ has integral trace, and all its other conjugates have absolute value strictly less than $1$.
 
\smallskip

\begin{theorem}\label{maintheorem} Let $\Gamma\subset \overline{\mathbb Q}^{\times}$ be a finitely generated multiplicative group of algebraic numbers.  Let $\alpha_1,\ldots, \alpha_r \in\overline{\mathbb Q}^\times$  be  $\mathbb{Q}$-linearly independent  algebraic numbers.  For a given real number $\varepsilon>0$,  let $\mathcal{B}$ be a subset of $\Gamma\times \mathbb{Z}^{r+1}$ which consists of  tuples $(u, q,  p_1,\ldots,p_r)\in\Gamma\times\mathbb{Z}^{r+1}$ with $d=[\mathbb{Q}(u):\mathbb{Q}]$ for some integer $d\geq 1$ satisfying    $|\alpha_i q u|>1$, $\alpha_i q u$  is not a pseudo-Pisot number  for some integer $i\in \{1,2,\ldots, r\}$   and
\begin{equation}\label{eq1.1}
0<|\alpha_j qu-p_j|<\frac{1}{H^\varepsilon(u)|q|^{\frac{d}{r}+\varepsilon}}\quad\mbox{for all integers~~} j=1,2,\ldots, r
\end{equation}
where $H(u)$ is the absolute Weil height of an algebraic number.  Then $\mathcal{B}$ is a finite set.  
\end{theorem}

When we put   $r=1$ in Theorem \ref{maintheorem}, we recover one of the main results of Corvaja and Zannier in \cite{corv}.  As an application of Theorem \ref{maintheorem}, we  prove the following  transcendence criterion.

\begin{corollary}\label{cor1}
Let $\alpha_1,\ldots,\alpha_r$ be $\Q$-linearly independent real numbers.  Let  $\alpha$ be an algebraic number of degree $d\geq 2$  such that   one of its conjugates  $\beta $  satisfies  $|\beta| >|\alpha| > 1$. For some $\eta >0$,  suppose there exist infinitely many tuples $(n,q,p_1,\ldots,p_r)\in \mathbb{Z}_{>0}^2\times \mathbb{Z}^r$  satisfying  
$$
0<\left|\alpha_jq\alpha^n- {p_j}\right|<\frac{1}{H^\eta(\alpha^n) q^{\frac{d}{r}+\eta}}\quad\mbox{ for all integers~}  j = 1, \ldots, r,
$$
where $H(u)$ is the absolute Weil height of an algebraic number.  Then at least one of the numbers among $\alpha_1, \ldots, \alpha_r$  is transcendental.
\end{corollary}
 Corollary \ref{cor1} is an extension of a result due to Han\v{c}l,  Kolouch, Pulcerov\'a and \v{S}t\v{e}pni\v{c}ka in \cite{hancl}.  It is important to note that   Corvaja and Han\v{c}l \cite{cH}, in 2007,   proved the transcendence of infinite product using the new diophantine approximation result proved  by  Corvaja and Zannier \cite{corv} for the first time in the literature.  As an example, with the similar spirits in \cite{hancl}, one can conclude the following transcendental result. 
 \smallskip
 
 {\it Let $\alpha > 1$ be a real algebraic number of degree $d\geq 2$ such that a conjugate $\beta$ of $\alpha$ satisfying $|\beta| >\alpha$ and let $r\geq 2$ be an integer. Let $\delta >0$ and $\epsilon > 0$ be real numbers satisfying $\displaystyle \frac{1+\frac{d}r + \delta}{\frac dr+1} \cdot \frac{\epsilon}{1+\epsilon} > 1$. Let $(a_n)_{n\geq 1}$ be a sequence of positive integers  and for $i = 1, 2, \ldots, r$, let  $(b^{(i)}_n)_{n\geq 1}$ be  a sequence of positive integers. Let $B_n^{(i)} = b_n^{(i)}\alpha^{a_n}$ for all $n\geq 1$ and for all $i = 1, 2, \ldots, r$. Suppose the sequence $(B_n^{(i)})_{n\geq 1}$ is non-decreasing  and satisfying the growth conditions 
 $$
 \limsup_{n\to\infty}\left(B_n^{(i)}\right)^{\frac{1}{(2+(d/r)+\delta)^n}} = \infty \mbox{ and } B_n^{(i)} \gg n^{1+\epsilon} 
 $$
 for all $i = 1, 2, \ldots, r$.  If  $\alpha_i = \displaystyle \prod_{n=1}^\infty\frac{[B_n^{(i)}]}{B_n^{(i)}}$ for all $i = 1, 2, \ldots, r$, then either at least one of the numbers among $\alpha_1, \ldots, \alpha_r$ is transcendental or they are $\mathbb{Q}$-linearly dependent.  } Here $[x]$ denotes the integral part of the real number $x$.

\bigskip

Let $K$ be a number field and  let $S$  be a finite set of places on $K$  such that $S$ contains all the archimedean valuations of $K$.  The group of $S$-units, denoted by $\mathcal{O}^\times_S$, is defined as
$$
\mathcal{O}^\times_S=\{\alpha\in K  :   |\alpha|_\mathit{v}=1\quad\mbox{for all~~}\mathit{v}\notin S\}.
$$
Now, we shall state the other result of this article as follows. 
\begin{theorem}\label{maintheorem1} 
Let $K$  be a number field of degree $n$ which is Galois  over $\mathbb{Q}$    and let $S$  be a finite set of places on $K$  such that $S$ contains all the archimedean places of $K$.   Let $d$ be a divisor of $n$  and  let $\alpha_1,\ldots,\alpha_d\in K$ be given algebraic numbers and not all zero. For a given real number $\varepsilon >0$, let $\mathcal{B}$ be a subset of $\mathcal{O}_S^\times\times\mathbb{Z}^2$ which consists of   triples $(u_1, q,  p)\in\mathcal{O}_S^\times\times\mathbb{Z}^2$ with $d=[\mathbb{Q}(u_1):\mathbb{Q}]$  for some integer $d\geq 1$ satisfying   $|\alpha_i q u_i|>1$, $\alpha_iq u_i$  is not a  pseudo-Pisot number for some integer $i\in\{1,2,\ldots,d\}$   and
\begin{equation*}
0<|\alpha_1 q u_1+ \alpha_2qu_2+\cdots+\alpha_d q u_d-p|\leq 
\frac{1}{H^\varepsilon(u_1)|q|^{d+\varepsilon}},
\end{equation*}
where  $u_1, u_2,\ldots, u_d$  are all Galois conjugates. Then $\mathcal{B}$ is a finite set. 
\end{theorem}

\section{Preliminaries} 

Let $K\subset \mathbb{C}$ be a  number field which is  Galois  over $\mathbb{Q}$ with Galois group $\mbox{Gal}(K/\mathbb{Q})$. Let $M_K$ be the set of all inequivalent places of $K$ and $M_\infty$ be the set of all archimedean places of $K$. For each place $v\in M_K$, we denote $|\cdot |_v$ the absolute value corresponding to $v$, normalized with respect to $K$. Indeed if $v\in M_\infty$, then there exists an automorphism $\sigma\in\mbox{Gal}(K/\mathbb{Q})$ of $K$ such that for all $x\in K$, 
\begin{equation}\label{eq2.1}
|x|_v=|\sigma(x)|^{d(\sigma)/[K:\mathbb{Q}]},
\end{equation}
where $d(\sigma) =1$ if $\sigma(K) = K\subset \mathbb{R}$ and $d(\sigma) = 2$ otherwise. Note that $d(\sigma)$ is constant since $K/\mathbb{Q}$ is Galois. Non-archimedean absolute values are normalized accordingly so that  the product formula   $\displaystyle\prod_{\omega\in M_K}|x|_\omega=1$ holds for any  $x\in K^\times$.

The absolute Weil  height  $H(x)$ is defined as 
$$
H(x):=\prod_{\omega\in M_K}\mbox{max}\{1,|x|_\omega\} \mbox{ for all } x\in K.
$$
For a vector $\textbf{x} =(x_1,\ldots,x_n)\in K^n$  and for a place $\omega\in M_K$, the $\omega$-norm for  $\textbf{x}$, denoted by $||\textbf{x}||_\omega$, is defined by 
$$
\Vert\textbf{x}\Vert_\omega:=\mbox{max}\{|x_1|_\omega,\ldots,|x_n|_\omega\}
$$
and  the projective height,  $H(\textbf{x})$, is defined by 
$$
H(\textbf{x})=\prod_{\omega\in M_K}\Vert\textbf{x}\Vert_\omega.
$$
Now we are ready to state a more general version  of the Schmidt Subspace Theorem, which was  formulated by    Evertse and Schlickewei.  For a reference,  see \cite[Chapter 7]{bomb},  \cite[Chapter V, Theorem $1D^\prime$]{schmidt}, and  \cite[Page 16, Theorem II.2]{zannier}.

\smallskip

\begin{theorem} (Subspace Theorem) \label{schli}
 Let $K$ be a number field and $m \geq 2$ be an integer. Let $S$ be a finite set of places on $K$ which contains all the archimedean  places of $K$.  For each $v \in S$, let $L_{1,v}, \ldots, L_{m,v}$ be linearly independent linear  forms in the variables $X_1,\ldots,X_m$ with  coefficients in $K$.  For any $\varepsilon>0$, the set of solutions $\mathbf{x} \in K^m$ to the inequality 
\begin{equation*}
\prod_{v\in S}\prod_{i=1}^{m} \frac{|L_{i,v}(\mathbf{x})|_v}{\Vert\mathbf{x}\Vert_v} \leq \frac{1}{H(\mathbf{x})^{m+\varepsilon}}
\end{equation*}
lies in finitely many proper subspaces of $K^m$.
\end{theorem}

The following lemma is an application of Theorem \ref{schli}  which can be deduced from  results obtained by Evertse. For a proof, we refer to \cite[Lemma 1]{corv}.

\begin{lemma}\label{lem2.1} 
Let $K$ be a number field which is  Galois over $\mathbb{Q}$  and $S$  be a finite subset of places which contains all the archimedean places.  Let $\sigma_1,\ldots,\sigma_n$  be distinct automorphisms of $K$ for some integer $n\geq 1$ and let $\lambda_1,\ldots,\lambda_n$  be non-zero elements of $K$. Let  $\varepsilon>0$  be a given real number and $\omega\in S$  be a distinguished place.  Let $\mathfrak{E}\subset \mathcal{O}_S^\times$ be a subset which is defined as  
$$
\mathfrak{E} := \left\{ u\in \mathcal{O}_S^\times \ :  \ |\lambda_1 \sigma_1(u)+\cdots+\lambda_n\sigma_n(u)|_\omega<\frac{1}{H^{\varepsilon}(u)}\max\{|\sigma_1(u)|_\omega,\ldots,|\sigma_n(u)|_\omega\}\right\}.
$$
If $\mathfrak{E}$ is an infinite subset of $\mathcal{O}_S^\times$, then there exist $a_1, \ldots, a_n\in K$, not all zero, such that 
$$
a_1\sigma_1(v)+\cdots+a_n\sigma_n(v)=0 
$$
holds true for infinitely many elements  $v\in \mathfrak{E}$.
\end{lemma}
We also need the following lemma, which is  a  special case of  the $S$-unit equation theorem proved by Evertse and van der Poorten-Schlickewei. For a proof, we refer to \cite[page 18, Theorem II.4]{zannier}. 
\begin{lemma}\label{lem2.2}
Let $K, S$  and $\sigma_1,\ldots,\sigma_n$  be as Lemma \ref{lem2.1}. Let  $a_1,\ldots,a_n$  be non-zero elements of $K$. Let $\mathfrak{E}\subset \mathcal{O}_S^\times$ be a  subset which is defined as 
$$
\mathfrak{E} := \left\{u\in \mathcal{O}_S^\times\ : \  a_1\sigma_1(u)+\cdots+a_n\sigma_n(u)=0 \right\}.
$$
If $\mathfrak{E}$ is an infinite set, then there exist  integers $i\neq j$ satisfying $1\leq i < j \leq n$, non-zero elements $a,b\in K^\times$ and infinitely many $v\in \mathfrak{E}$  satisfying 
$$
a\sigma_i(v)+b\sigma_j(v)=0.
$$
\end{lemma}

We shall start with the following observation.

\begin{lemma}\label{newlemma}
Let $K$ be a number field  of degree $n$ which is  Galois  over $\mathbb{Q}$  and $k\subset K$  be a subfield of degree $d$  over $\mathbb{Q}$ for some integer $d\geq 2$. Let $\alpha_1,\ldots,\alpha_r$ be  $\mathbb{Q}$-linearly independent elements of $K$ for some integer $r\geq 1$.  Let $S$  be a finite set of places on $K$ which contains all the archimedean places and  let $\varepsilon>0$   be a given real number.   Let 
\begin{equation}\label{eq2.3}
\mathcal{B} = \left\{(u, q, p_1,\ldots,p_r)\in (\mathcal{O}_S^\times\cap k)\times\mathbb{Z}^{r+1} \ : \   
0<|\alpha_i qu-p_i|<\frac{1}{H^\varepsilon(u)|q|^{\frac{d}{r}+\epsilon}} \mbox{ for all } 1\leq i\leq r\right\}
\end{equation}
be a subset of $(\mathcal{O}_S^\times\cap k)\times\mathbb{Z}^{r+1}$. If $\mathcal{B}$ is an infinite set, then $H(u) \to\infty$ as $u$ varies over all the tuples $(u,q, p_1, \ldots, p_r) \in \mathcal{B}$.
 \end{lemma} 
 \begin{proof}
 If possible,  $H(u)$ is bounded as $u$ varies over all the tuples $(u, q, p_1, \ldots, p_r) \in\mathcal{B}$.  Then there exists a fixed $u$, say,   $u_0$ and an infinite subset $\mathcal{A}$ of $\mathcal{B}$ such that if  $(u, q, p_1, \ldots, p_r) \in \mathcal{A}$, then $u = u_0$ and satisfying  
\begin{equation}\label{eq2.4}
0<|\alpha_i q u_0-p_i|<\frac{1}{H^{\epsilon}(u_0)|q|^{\frac{d}{r}+\varepsilon}},\quad\mbox{for all~~}1\leq i\leq r
\end{equation}
holds true for all tuples $(u_0,q,p_1,\ldots,p_r)\in\mathcal{A}$. Since $H(u_0)$ is a constant and if $d/r\geq 1$,  then, by Roth's theorem,  we conclude that $\alpha_1 u_0, \ldots,\alpha_r u_0$  are all transcendental which is   a contradiction.   Hence,  we  can assume that $d/r<1$.  

The equation \eqref{eq2.4} implies that $\alpha_1u_0, \ldots, \alpha_ru_0$ have simultaneous rational approximation with common denominator $q$ whose exponent is $1+d/r +\epsilon = 1+\delta$ where $\delta \geq 1/r$. By the well-known application of the subspace theorem on simultaneous approximation with common denominator, one gets either one fo the numbers $\alpha_1u_0, \ldots, \alpha_ru_0$ is  transcendental or $1, \alpha_1u_0, \ldots, \alpha_ru_0$ are $\mathbb{Q}$-linearly dependent.  Since $\alpha_iu_0$ is algebraic for each integer $i\geq 1$, we conclude that $1, \alpha_1u_0, \ldots, \alpha_ru_0$ are $\mathbb{Q}$-linearly dependent. This does not contradict to the fact that $\alpha_1u_0, \ldots, \alpha_ru_0$ are $\mathbb{Q}$-linearly independent. To get a contradiction, we proceed as follows.
\smallskip

Now consider $r+1$ linearly independent linear forms with algebraic coefficients as   
$$
L_{0,\infty}(x_0, x_1,\ldots,x_r)=x_0 \mbox{ and } 
L_{i,\infty}(x_0, x_1,\ldots,x_r)=\alpha_i u_0 x_0-x_i 
$$
for all $i = 1,2,\ldots, r$. We  take  $K = \mathbb{Q}$ and $S = \{\infty\}$ in Theorem \ref{schli}.  Then by  \eqref{eq2.4}, we conclude that there are infinitely many integer tuples  $(q, p_1,p_2,\ldots ,p_r)$ where $(u_0, q, p_1,\ldots, p_r) \in\mathcal{A}$ satisfying Theorem \ref{schli}.  Therefore there exist $a_0, a_1, \ldots, a_r \in \mathbb{Z}$, not all zero, such that  \begin{equation}\label{eq2.5}
a_0 q+a_1 p_1+\cdots+a_r p_r=0
\end{equation}
holds true for all tuples $(u_0, q, p_1,\ldots, p_r) \in\mathcal{A}'$ where $\mathcal{A}'$ is an infinite subset of $\mathcal{A}$. Since not all $a_i$'s are  $0$   in \eqref{eq2.5}, we assume that   $a_{i_0}\neq 0$ for some integer $i_0$ satisfying $1\leq i_0\leq r$.
 
\bigskip

\noindent{\bf Claim.} There exist an infinite subset $\mathcal{A}''$ of $\mathcal{A}'$ and integers $b_1, \ldots, b_r$ (not all are zero) such that $b_1p_1+\cdots +b_rp_r = 0$ for all the tuples $(u_0, q, p_1,\ldots, p_r) \in \mathcal{A}''$ 

\bigskip

  Since  $d/r<1$ and $r, d\geq 2$, we see that $\displaystyle\frac{(r-1)d}{r}>1$.  We consider $r$ linearly independent linear forms with algebraic coefficients as 
$$
L_{i,\infty}(x_1,x_2, \ldots ,x_r)=\alpha_i u_0 x_{i_0}-x_i  \mbox{  for all  }  i = 1,2,\ldots, i_0-1, i_0+1,\ldots, r, \mbox{ and }  L_{i_0,\infty}(x_1,x_2, \ldots ,x_r)=x_{i_0}.
$$
 Now, we take $K = \mathbb{Q}$ and $S = \{\infty\}$ in Theorem \ref{schli}.  In order to conclude the assertion of Theorem \ref{schli}, we need to estimate the following quantity
$$
\prod_{i=1}^r|L_{i,\infty}(p_1,p_2,\ldots,p_{i_0-1},q, p_{i_0+1},\ldots,p_{r})| 
$$
for all the integer tuples $(p_1,\ldots, p_{i_0-1}, q, p_{i_0+1}, \ldots ,p_r)$ where $(u_0, q$, $p_1,\ldots, p_r) \in\mathcal{A}'$. By 
 \eqref{eq2.4}, we get 
$$
\prod_{i=1}^r|L_{i,\infty}(p_1,p_2,\ldots,p_{i_0-1},q, p_{i_0+1},\ldots,p_{r})|<\frac{1}{H^\varepsilon(u_0)}\frac{|q|}{|q|^{\frac{(r-1)d}{r}+\varepsilon}}
$$
holds true for all tuples $(u_0,q,p_1,\ldots,p_r)\in\mathcal{A}'$.  Since  $(r-1)d/r>1$,  we conclude that 
$$
\prod_{i=1}^r|L_{i,\infty}(p_1,p_2,\ldots,p_{i_0-1},q, p_{i_0+1},\ldots,p_{r})|<\frac{1}{H^\varepsilon(u_0)}\frac{|q|}{|q|^{\frac{(r-1)d}{r}+\varepsilon}}\leq \frac{1}{H^\varepsilon(u_0)}\frac{1}{|q|^{\varepsilon}}
$$
holds true for  all tuples $(u_0, q$, $p_1,\ldots, p_r) \in\mathcal{A}'$.   Therefore, by Theorem \ref{schli},  we get a non-trivial relation  
\begin{equation}\label{eq2.6}
 b_{i_0} q +b_1 p_1+\cdots+b_{i_0-1}p_{i_0-1}+b_{i_0+1}p_{i_0+1}+\cdots+b_r p_r=0\quad\mbox{where}\quad b_i\in\Z
\end{equation}
holds true for all  tuples $(u_0, q, p_1,\ldots, p_r) \in \mathcal{A}''$ for some infinite subset $\mathcal{A}''$ of $\mathcal{A}'$. 
Since $\alpha_1, \ldots, \alpha_{i_0-1}$, $\alpha_{i_0+1}$, $\ldots$, $\alpha_r$ are $\mathbb{Q}$-linearly independent, we conclude that $b_{i_0} \ne 0$.  
Now, by \eqref{eq2.5} and \eqref{eq2.6}, we can arrive at relation 
 $$
(b_{i_0} a_1-a_0 b_1) p_1+(b_{i_0} a_2-a_0 b_2)p_2+\cdots+(b_{i_0} a_{i_0-1}-a_0 b_{i_0-1})p_{i_0-1}+b_{i_0} a_{i_0} p_{i_0}+\cdots+(b_{i_0} a_r-a_0 b_r)p_r=0 
$$
which holds true for all tuples $(p_1,\ldots,p_r)$ where $(u_0,q,p_1,\ldots,p_r)\in\mathcal{A}''$. Since $a_{i_0}$ and $b_{i_0}$ are non-zero, we conclude that the above relation  is non-trivial. This proves the claim.

Now,  dividing by $q$ and letting the tuples  $(p_1, \ldots, p_r)$ vary over all  tuples in $\mathcal{A}''$, we get 
  $\alpha_1 u_0,\ldots,\alpha_r u_0$ are $\Q$-linearly dependent, which is a contradiction as $\alpha_1,\alpha_2,\ldots,\alpha_r$ are $\Q$-linearly independent.   Hence, $H(u) \to \infty$, as desired. 
 \end{proof}

The following lemmas are key to prove  Theorems \ref{maintheorem} and \ref{maintheorem1}.

\begin{lemma}\label{lem3.1}
Let $K$ be a number field of degree $n$ which is  Galois  over $\mathbb{Q}$   and $k\subset K$  be a subfield of degree $d$  over $\mathbb{Q}$. Let $\alpha_1,\ldots,\alpha_r$ be  $\mathbb{Q}$-linearly independent elements of $K$ for some integer $r\geq 1$.  Let $S$  be a finite set of places on $K$ which contains all the archimedean places and  let $\varepsilon>0$   be  a given real number.   Let $\mathcal{B}$ be a subset of  $(\mathcal{O}_S^\times\cap k)\times\mathbb{Z}^{r+1}$ as defined in  \eqref{eq2.3} such that for each tuple $(u, q, p_1, \ldots, p_r) \in \mathcal{B}$, there exists an integer $i \in \{1, \ldots, r\}$ satisfying $|q\alpha_i u|>1$ and $q\alpha_iu$ is not a  pseudo Pisot number.   If $\mathcal{B}$ is infinite, then there exist a proper subfield $k'\subset k$, a non-zero element $u'\in k$  and an infinite subset $\mathcal{B}'\subset \mathcal{B}$  such that $u/u'\in k'$ for all $(u, q, p_1, p_2,\ldots, p_r)\in\mathcal{B}'$. 
\end{lemma}


\begin{proof}
First note that $d\geq 2$ as $\Q$ doesn't admit any proper subfield in it.  Let $\mathcal{G} = \mbox{Gal}(K/\mathbb{Q})$ be the Galois group of $K$ over $\mathbb{Q}$. Since $K$ over $k$ is Galois, we let  $\mathcal{H}:=\mbox{Gal}(K/k)\subset \mathcal{G}$  be the subgroup fixing $k$.  Hence, $|\mathcal{G}/\mathcal{H}| = [k:\mathbb Q] = d$. Therefore, among the $n$ embeddings of $K$, there are   $d$ embeddings, say, $Id=\sigma_1,\ldots,\sigma_d$  which are the complete set of  representatives for the left cosets of $\mathcal{H}$ in $\mathcal{G}$ and more precisely, we have 
$$
\mathcal{G}/\mathcal{H}:=\{\mathcal{H}, \sigma_2 \mathcal{H},\ldots,\sigma_d \mathcal{H}\}.
$$
Each $\rho\in \mathcal{G}$ defines an archimedean valuation on $K$ by the formula
\begin{equation}\label{newarchimedean}
|\alpha|_\rho := |\rho^{-1}(\alpha)|^{d(\rho)/[K:\mathbb{Q}]},
\end{equation}
where $|\cdot|$ denotes the usual absolute value in $\mathbb{C}$. Two elements $\rho_1 \ne \rho_2$ in $\mathcal{G}$ define the same valuation if and only if $\rho_1^{-1}\circ \rho_2$ is the complex conjugation.  Then for a fixed $i$ with $1\leq i\leq r$,  by \eqref{newarchimedean}, for each $\rho\in\mbox{Gal}(K/\mathbb{Q})$ and for each tuple $(u, q, p_1,\ldots, p_r) \in \mathcal{B}$,  we have, 
\begin{equation}\label{eq2.7}
|\alpha_i qu-p_i|^{d(\rho)/[K:\mathbb{Q}]}=|\rho(\alpha_i)\rho(qu)-\rho(p_i)|_\rho=|\rho(\alpha_i)q\rho(u)-p_i|_\rho.
\end{equation}
For each $\mathit{v}\in M_\infty$, let $\rho_\mathit{v}$ be an automorphism defining the valuation $\mathit{v}$, according to \eqref{newarchimedean}: $|\alpha|_\mathit{v}:= |\alpha|_{\rho_v}$; Then the set $\{\rho_v : v\in M_\infty\}$ represents the left cosets of the subgroup generated by the complex conjugation in $\mathcal{G}$.  
For each $j = 1, 2, \ldots, d$,  let 
$$S_j = \left\{v\in M_\infty\ : \ \rho_v\vert_k = \sigma_j: k\rightarrow \mathbb{C}\right\}
$$
and hence $S_1\cup\ldots\cup S_d=M_\infty$.   Thus, we have $M_\infty = \{\rho_v : v\in M_\infty\}$ and by \eqref{eq2.7}, we get 
\begin{equation}\label{eq2.8}
\prod_{\mathit{v}\in M_\infty}|\rho_\mathit{v}(\alpha_i)\rho_{\mathit{v}}(qu)-p_i|_\mathit{v}=\prod_{j=1}^d\prod_{\mathit{v}\in S_j}|\rho_{\mathit{v}}(\alpha_i)\sigma_j(qu)-p_i|_v.
\end{equation} 
By \eqref{eq2.7}, we see that 
$$
\prod_{\mathit{v}\in M_\infty}|\rho_\mathit{v}(\alpha_i)\rho_{\mathit{v}}(qu)-p_i|_\mathit{v}=\prod_{\mathit{v}\in M_\infty}|\alpha_i q u-p_i|^{d(\rho_\mathit{v})/[K:\mathbb{Q}]}=|\alpha_i q u-p_i|^{{\sum_{\mathit{v}\in M_\infty}}d(\rho_\mathit{v})/[K:\mathbb{Q}]}.
$$
Then, from \eqref{eq2.8} and the well-known formula $\displaystyle\sum_{\mathit{v}\in M_\infty}d(\rho_\mathit{v})=[K:\mathbb{Q}]$, it follows that 
\begin{equation}\label{eq2.9}
\prod_{j=1}^d\prod_{\mathit{v}\in S_j}|\rho_{\mathit{v}}(\alpha_i)\sigma_j(qu)-p_i|_v=|\alpha_i q u-p_i|
\end{equation}
for all integers $i =  1, \ldots,  r.$   

Now, for each $\mathit{v}\in S$, we define $d+r$ linearly independent linear forms in $d+r$  variables as follows: For $j = 1, 2, \ldots, d$ and for  $\mathit{v}\in S_j$ and for each integer $i$ satisfying $1\leq i\leq r$, we let  
$$
L_{\mathit{v},i}(x_1,\ldots,x_r,\ldots,x_{r+d}) = x_i-\rho_{\mathit{v}}(\alpha_i) x_{j+r}, 
$$
and when the  integer $i$ in the range   $r+1\leq i\leq d+r$, we let 
$$
L_{v, i}(x_1,\ldots,x_r,\ldots,x_{r+d})=x_i.
$$
For each  $\mathit{v}\in S\backslash{M_\infty}$ and for  each integer $i$ satisfying $1\leq i \leq r+d$,  we let  
$$
L_{v,i}(x_1,\ldots,x_r,\ldots,x_{r+d})=x_i.
$$
Let $\mathbf{X}$ be the element in $K^{d+r}$  of the form
$$
\mathbf{X}=(p_1,p_2,\ldots,p_r,q\sigma_1(u),\ldots,q\sigma_d(u)) \in K^{d+r}.
$$
In order to apply Theorem \ref{schli}, we need to estimate the following quantity 
\begin{equation}\label{eq2.10}
\prod_{\mathit{v\in S}}\prod_{j=1}^{d+r}\frac{|L_{\mathit{v},j}(\mathbf{X})|_\mathit{v}}{\Vert\mathbf{X}\Vert_\mathit{v}}.
\end{equation}
Using the fact that $L_{\mathit{v},j}(\mathbf{X})=q\sigma_j(u)$, for $r+1\leq j\leq d+r$, we obtain
\begin{eqnarray*}
\prod_{\mathit{v\in S}}\prod_{j=r+1}^{d+r}|L_{\mathit{v},j}(\mathbf{X})|_\mathit{v}&=&\prod_{\mathit{v}\in S}\prod_{j=r+1}^{d+r}|q\sigma_j(u)|_{\mathit{v}} 
=\prod_{\mathit{v}\in S}\prod_{j=r+1}^{d+r}|q|_\mathit{v}\prod_{j=r+1}^{d+r}\prod_{\mathit{v}\in S}|\sigma_j(u)|_\mathit{v}.
\end{eqnarray*}
Since $\sigma_j(u)$  are $S$-units,  by the product formula,  we obtain
$$
\prod_{\mathit{v}\in S}|\sigma_j(u)|_\mathit{v}=\prod_{\mathit{v}\in M_K}|\sigma_j(u)|_\mathit{v}=1.  
$$
Consequently, the above equality implies 
$$
\prod_{\mathit{v\in S}}\prod_{j=r+1}^{d+r}|L_{\mathit{v},j}(\mathbf{X})|_\mathit{v}=\prod_{\mathit{v}\in S}\prod_{j=r+1}^{d+r}|q|_\mathit{v}\leq \prod_{v\in M_\infty}\prod_{j=r+1}^{d+r}|q|_\mathit{v}=\prod_{j=r+1}^{d+r}|q|^{\sum_{\mathit{v}\in M_\infty}d(\rho_\mathit{v})/[K:\mathbb{Q}]}.
$$
Then, from the formula $\sum_{\mathit{v}\in M_\infty}d(\rho_\mathit{v})=[K:\mathbb{Q}]$, we get 
\begin{equation}\label{eq2.11}
\prod_{\mathit{v\in S}}\prod_{j=r+1}^{d+r}|L_{\mathit{v},j}(\mathbf{X})|_\mathit{v}\leq \prod_{v\in M_\infty}\prod_{j=r+1}^{d+r}|q|_\mathit{v}=\prod_{j=r+1}^{d+r}|q|^{\sum_{\mathit{v}\in M_\infty}d(\rho_\mathit{v})/[K:\mathbb{Q}]}=|q|^d.
\end{equation}
Now we estimate the denominators of the product in \eqref{eq2.10} as follows; We have  
$$
\prod_{\mathit{v\in S}}\prod_{j=1}^{d+r}\Vert\mathbf{X}\Vert_\mathit{v}\geq \prod_{\mathit{v\in M_K}}\prod_{j=1}^{d+r}\Vert\mathbf{X}\Vert_\mathit{v}=\prod_{j=1}^{d+r}\left(\prod_{\mathit{v\in M_K}}\Vert\mathbf{X}\Vert_\mathit{v}\right) = \prod_{j=1}^{d+r}H(\mathbf{X}),
$$
since $\Vert\mathbf{X}\Vert_\mathit{v}\leq 1$ for all $\mathit{v}\not\in S$. Thus, we get, 
\begin{equation}\label{eq2.12}
\prod_{\mathit{v\in S}}\prod_{j=1}^{d+r}\Vert\mathbf{X}\Vert_\mathit{v}\geq H(\mathbf{X})^{d+r}.
\end{equation}
By \eqref{eq2.10}, \eqref{eq2.11}  and \eqref{eq2.12}, it follows that 
$$
\prod_{\mathit{v\in S}}\prod_{j=1}^{d+r}\frac{|L_{\mathit{v},j}(\mathbf{X})|_\mathit{v}}{\Vert\mathbf{X}\Vert_\mathit{v}}\leq \frac{1}{H^{d+r}(\mathbf{X})}|q|^d\prod_{i=1}^r|\alpha_i q u-p_i|.
$$
Thus, from \eqref{eq2.3}, we have
$$
\prod_{\mathit{v\in S}}\prod_{j=1}^{d+r}\frac{|L_{\mathit{v},j}(\mathbf{X})|_\mathit{v}}{||\mathbf{X}||_\mathit{v}}\leq \frac{1}{H^{d+r}(\mathbf{X})}|q|^d\frac{1}{H^{r\varepsilon}(u)}\frac{1}{|q|^{d+r\epsilon}}=\frac{1}{H^{d+r}(\mathbf{X})}\frac{1}{(|q|H(u))^{r\varepsilon}}.
$$
Notice that 
\begin{eqnarray*}
H(\mathbf{X})&=&\prod_{\mathit{v}\in M_K}\mbox{max}\{|p_1|_\mathit{v},\ldots,|p_r|_\mathit{v},|q\sigma_1(u)|_\mathit{v},\ldots,|q\sigma_d(u)|_\mathit{v}\}\\
&=&\prod_{\mathit{v}\in S}\mbox{max}\{|p_1|_\mathit{v},\ldots,|p_r|_\mathit{v},|q\sigma_1(u)|_\mathit{v},\ldots,|q\sigma_d(u)|_\mathit{v}\}\\
&\leq& \prod_{\mathit{v}\in S}\mbox{max}\{|qp_1\cdots p_r|_\mathit{v},\ldots,|q p_1\cdots p_r|_\mathit{v},|q p_1\cdots p_r\sigma_1(u)|_\mathit{v},\ldots,|qp_1\cdots p_r\sigma_d(u)|_\mathit{v}\}\\
&\leq&|qp_1\cdots p_r|\prod_{\mathit{v}\in S}\mbox{max}\{1,|\sigma_1(u)|_\mathit{v},\ldots,|\sigma_d(u)|_\mathit{v}\}\\ 
&\leq& |qp_1\cdots p_r|\left(\prod_{\mathit{v}\in S}\mbox{max}\{1,|\sigma_1(u)|_\mathit{v}\}\right)\cdots \left(\prod_{\mathit{v}\in S}\mbox{max}\{1,|\sigma_d(u)|_\mathit{v}\}\right)\\
&=&|qp_1\cdots p_r|H^d(u).
\end{eqnarray*} 
By using the the inequality 
$
||x|-|y||\leq |x-y|
$
and the fact, by Lemma \ref{newlemma}, that $H(u)\to \infty$ when $u$ varies over all the tuples $(u, q, p_1, \ldots, p_r)\in \mathcal{B}$, by \eqref{eq2.3}, we conclude that $|p_i|\leq |\alpha_i q u|+1.$  Since $|u|^{\frac{1}{d}}\leq H(u)$,   for all integers $i$ satisfying $1\leq i\leq r$, we get   
$$
|p_i|\leq |\alpha_i q u|+1\leq |q||\alpha_i|H^d(u)+1\leq |q|H^{2d}(u)
$$
holds true for all but finitely many tuples $(u, q, p_1, \ldots, p_r)\in \mathcal{B}$. By combining both these observations, we obtain $H(\mathbf{X})\leq |q|^{r+1} H(u)^{2rd +d}$, and hence  we get $H(\mathbf{X})^{1/(d(2r+1))} \leq |q|H(u)$. Therefore, 
$$
\prod_{\mathit{v\in S}}\prod_{j=1}^{d+r}\frac{|L_{\mathit{v},j}(\mathbf{X})|_\mathit{v}}{||\mathbf{X}||_\mathit{v}}\leq \frac{1}{H^{d+r}(\mathbf{X})}\frac{1}{(|q|H(u))^{r\varepsilon}}\leq \frac{1}{H(\mathbf{X})^{d+r+(r\varepsilon)/(2rd+1)}} = \frac{1}{H(\mathbf{X})^{d+r+ \varepsilon'}},
$$
for some $\varepsilon' >0$ holds true for infinitely many tuples $(u, q,p_1,\ldots,p_r)\in\mathcal{B}$.  Then, by Theorem \ref{schli}, there exists a proper subspace of $K^{d+r}$ that contain infinitely many $\mathbf{X} \in \mathcal{B}$.  
That is,  we have a non-trivial  relation  
\begin{equation}\label{eq2.13}
a_1 p_1+a_2 p_2+\cdots+a_r p_r+b_1q\sigma_1(u)+\cdots+b_d q \sigma_d(u)=0,\quad a_i, b_j\in K,
\end{equation}
holds true for all the tuples $(u, q, p_1,\ldots,p_r)\in\mathcal{B}_1$  for some infinite subset $\mathcal{B}_1$ of $\mathcal{B}$.   

\bigskip

By the hypothesis, we know that   $\alpha_i q u$ is not a  pseudo-Pisot number for some integer $i$.  Without loss of generality,  we can assume that for each $(u,q,p_1, \ldots, p_r) \in \mathcal{B}_1$, we have   $|\alpha_r qu|>1$ and $\alpha_rqu$  is not a  pseudo-Pisot number. Under the same hypothesis as in Lemma 3 in \cite{corv}, Corvaja and Zannier established a relation of the form $a_1 p+ b_1 q \sigma_1(u)+\cdots+b_d q \sigma_d(u)=0$. Therefore,  in view their  work, it is enough to show the existence of a non-trivial linear relation as in \eqref{eq2.13}  with $a_1=a_2=\cdots=a_{r-1}=0$.

\bigskip

\noindent{\bf Claim 1.}  At least one of the $b_j$'s is non-zero in the relation \eqref{eq2.13}.

\bigskip

If possible, suppose $b_i=0$  for all integers $i = 1, 2, \ldots, r$.  Then from \eqref{eq2.13}, we get 
\begin{equation}\label{eq2.14}
a_1 p_1+a_2 p_2+\cdots+a_r p_r=0\quad\mbox{with~~} a_i\in K.
\end{equation}
If  all the $a_i$'s are rational numbers, then, dividing  by $qu$, we obtain 
\begin{equation}\label{new1000}
a_1\frac{p_1}{qu}+\cdots+a_r\frac{p_r}{qu}=0
\end{equation}
holds true for all tuples $(u,q, p_1,\ldots,p_r) \in \mathcal{B}_1$.   For each such tuple,  the inequality 
$$
0<|\alpha_i qu-p_i|<\frac{1}{H^\varepsilon(u)|q|^{\frac{d}{r}+\epsilon}}  \iff 0< \left|\alpha_i - \frac{p_i}{qu}\right| < \frac{1}{H^\varepsilon(u)|u||q|^{1+\frac{d}{r}+\epsilon}} 
$$
holds true for every $i = 1, \ldots, r$.     As the tuples $(u, q, p_1, \ldots, p_r)$ vary over all the elements in $\mathcal{B}_1$,   by Lemma \ref{newlemma},  we have $H(u)\to\infty$. Therefore,  we conclude,  by \eqref{new1000},   that 
$$
a_1 \alpha_1 +\cdots+a_r \alpha_r =0.
$$
Since not all $a_i$'s are $0$, this  is a contradiction to the fact that $\alpha_i$'s are  $\mathbb{Q}$-linearly independent.   Hence we conclude that at  least one of $a_i$'s is algebraic irrational. Also, the sequence $\displaystyle\left(\frac{p_1}{qu}, \ldots, \frac{p_r}{qu}\right)$ tends to $(\alpha_1, \ldots \alpha_r)$ and since none of the $\alpha_i$'s are zero,  there exists an infinite subset $\mathcal{B}_2$ of $\mathcal{B}_1$ such that for any tuple $(u, q, p_1, \ldots, p_r) \in \mathcal{B}_2$ satisfying  $p_i \ne 0$ for all integers $i = 1, 2,\ldots, r$. 

\bigskip

Let  $\{a_1,a_2,\ldots,a_m\}$ be the maximal $\mathbb{Q}$-linearly independent subset of the set $\{a_1,\ldots, a_r\}$, if necessary, by renaming the indices.  Then we can write
$$
a_{m+i}=c_{i1}a_1+\cdots+c_{im}a_m, \mbox{~~where~~} c_{im}\in\mathbb{Q}, \quad\mbox{for all~~} 1\leq i\leq r-m.
$$
Thus by substituting the values of $a_{m+i}$ in \eqref{eq2.14}, we get 
$$
a_1 p_1+\cdots+a_m p_m+(c_{11}a_1+\cdots+c_{1m}a_m)p_{m+1}\cdots+(c_{(r-m)1}a_1+\cdots+c_{(r-m)m}a_m) p_r=0.
$$
We rewrite  this equality  in the following form
\begin{equation}\label{eq2.15}
a_1 (p_1+c_{11}p_{m+1}+\cdots+c_{(r-m)1}p_r)+\cdots+a_m(p_m+c_{1m}p_{m+1}+\cdots+c_{(r-m)m }p_r)=0.
\end{equation}
Since  $a_1,\ldots, a_m$ are $\mathbb{Q}$-linearly independent,  by  \eqref{eq2.15}, we obtain 
\begin{equation}\label{eq2.16}
p_1+c_{11}p_{m+1}+\cdots+c_{(r-m)1}p_r=0=\cdots=p_m+c_{1m}p_{m+1}+\cdots+c_{(r-m)m }p_r.  
\end{equation}
From \eqref{eq2.16}, we get a relation of the form \eqref{eq2.14} with rational coefficients, which is again not possible as observed earlier. Thus  this proves  Claim 1. 
%
%

\bigskip

\noindent{\bf Claim 2.}  There exists a non-trivial relation as in \eqref{eq2.13} with $a_i =0$ for all $i=1,2, \ldots, r-1$. 

\bigskip

We first prove that there exists a relation as in \eqref{eq2.13} with $a_1=0$. If possible,  we assume that $a_1\neq 0$. Then by rewriting the  relation \eqref{eq2.13} we obtain
\begin{equation}\label{eq2.17}
p_1=-\frac{a_2}{a_1}p_2-\cdots-\frac{a_r}{a_1}p_r-\frac{b_1}{a_1}q\sigma_1(u)-\cdots-\frac{b_d}{a_1}q\sigma_d(u).
\end{equation}

\noindent{\bf Case 1.} $\displaystyle\sigma_j\left(\frac{b_1}{a_1}\right) \ne \frac{b_j}{a_1}$ for some integer $j$ satisfying $2\leq j\leq d$. 

\smallskip

 By applying the automorphism $\sigma_j$  on both sides of \eqref{eq2.17}, we get
$$
p_1=-\sigma_j\left(\frac{a_2}{a_1}\right)p_2-\cdots-\sigma_j\left(\frac{a_r}{a_1}\right)p_r-\sigma_j\left(\frac{b_1}{a_1}\right)q\sigma_j\circ\sigma_1(u)-\cdots-\sigma_j\left(\frac{b_d}{a_1}\right)q\sigma_j\circ\sigma_d(u).
$$
By subtracting this relation with \eqref{eq2.17}, we get  a relation involving the terms only with  $p_2$,$\ldots$, $p_r$, $\sigma_1(u)$, $\ldots$, $\sigma_d(u)$.   Such a relation is non-trivial,  as the coefficient of $\sigma_j(u)$  is $\sigma_j(b_1/a_1)-\frac{b_j}{a_1}\neq 0$. 

\bigskip

\noindent{\bf Case 2.}  $\displaystyle\frac{b_j}{a_1}= \sigma_j\left(\frac{b_1}{a_1}\right)$  for all integers $j = 2, 3, \ldots, d$.  

\smallskip

Note that $b_1 \ne 0$. If not, then $0 =\sigma_j(b_1/a_1) = b_j/a_1$ for every integer $j$ and  hence  $b_i=0$ for all integers $i$, which contradicts Claim 1.  By putting $\lambda = -b_1/a_1$,   we rewrite \eqref{eq2.17} as 
\begin{equation}\label{eq2.18}
p_1=-\frac{a_2}{a_1}p_2-\cdots-\frac{a_r}{a_1}p_r+q(\sigma_1(\lambda)\sigma_1(u)+\cdots+\sigma_d(\lambda)\sigma_d(u)).
\end{equation}
 Clearly $a_i, b_j \in K$, but  it is not necessary that $\lambda$  belongs to $k$. If $\lambda\not\in k$,  then there exists an automorphism $\tau\in \mathcal{H}$ with $\tau(\lambda)\neq \lambda$.  By applying the automorphism $\tau$ on both sides of \eqref{eq2.18} and subtract with \eqref{eq2.18} to eliminate $p_1$, we obtain  
$$
p_2(\tau(a_2/a_1)-a_2/a_1)+\cdots+p_r(\tau(a_r/a_1)-a_r/a_1)+(\lambda-\tau(\lambda))\sigma_1(u)+\sum_{i=2}^d(\sigma_i(\lambda)\sigma_i(u)-\tau \circ\sigma_i(u))=0.
$$
Note that $\tau \circ\sigma_j$  coincides on $k$ with $\sigma_i$ for some integer $i$ and  since $\tau\in \mathcal{H}$  and $\sigma_2,\ldots,\sigma_d \not\in \mathcal{H}$, none of the $\tau \circ \sigma_j$ with $j\geq 2$ belongs in $\mathcal{H}$.  Hence the above relation can be written as a linear combination of $p_2,\ldots,p_r$ and  $\sigma_i(u)$ with the property that the coefficient of $\sigma_1(u)$  is $\lambda-\tau(\lambda)\ne 0$.   Therefore, we obtain a non-trivial linear relation among the $p_2,\ldots,p_r$ and $\sigma_i(u)$.

\smallskip

If  $\lambda\in k$,  then by adding $-\alpha_1 qu$ on both sides of  the equality   \eqref{eq2.18},  we get   
$$
|p_1-\alpha_1 qu|=\left|-\frac{a_2}{a_1}p_2-\cdots-\frac{a_r}{a_1}p_r+(\lambda-\alpha_1)q\sigma_1(u)+q\sigma_2(\lambda)\sigma_2(u)+\cdots+q\sigma_d(\lambda)\sigma_d(u)\right|.
$$
Then from \eqref{eq2.3}, we get
\begin{equation}\label{eq2.19}
\left|-\frac{a_2}{a_1}p_2-\cdots-\frac{a_r}{a_1}p_r+(\lambda-\alpha_1)q\sigma_1(u)+q\sigma_2(\lambda)\sigma_2(u)+\cdots+q\sigma_d(\lambda)\sigma_d(u)\right|<\frac{1}{H^\varepsilon(u)q^{\frac{d}{r}+\varepsilon}}.
\end{equation}
Then just like the inequality in \eqref{eq2.9}, we have the following important observation:
\begin{align*}
\prod_{j=1}^d\prod_{\mathit{v}\in S_j}&\left|-\rho_\mathit{v}\left(\frac{a_2}{a_1}\right)p_2-\cdots-\rho_\mathit{v}\left(\frac{a_r}{a_1}\right)p_r+(\rho_\mathit{v}(\lambda)-\rho_\mathit{v}(\alpha_1))q(\rho_\mathit{v}\circ\sigma_1)(u)+\cdots+q(\rho_\mathit{v}\circ\sigma_d)(\lambda)(\rho_\mathit{v}\circ\sigma_d)(u)\right|_{\mathit{v}}\\
&=\left|-\frac{a_2}{a_1}p_2-\cdots-\frac{a_r}{a_1}p_r+(\lambda-\alpha_1)q\sigma_1(u)+q\sigma_2(\lambda)\sigma_2(u)+\cdots+q\sigma_d(\lambda)\sigma_d(u)\right|.
\end{align*}
For each $\mathit{v}\in M_\infty$ and $j=1,\ldots,d$, we define $\mathit{v}(j)$ such that $\rho_\mathit{v}\circ \sigma_j=\sigma_{\mathit{v(j)}}$ on the field $k$, where $\{\mathit{v}(1),\ldots,\mathit{v}(d)\}$ is a permutation of $\{1,\ldots,d\}$.  Hence the above relation can be written as a linear combination of $p_2,\ldots,p_r$ and $\sigma_{\mathit{v}(1)},\ldots,\sigma_{\mathit{v}(d)}$. Therefore there exist   algebraic numbers  $c_{\mathit{v}(1)},\ldots,c_{\mathit{v}(d)}\in K$, not all zero, such that 
\begin{align*}
\prod_{j=1}^d\prod_{\mathit{v}\in S_j}&\left|-\rho_\mathit{v}\left(\frac{a_2}{a_1}\right)p_2-\cdots-\rho_\mathit{v}\left(\frac{a_r}{a_1}\right)p_r+(\rho_\mathit{v}(\lambda)-\rho_\mathit{v}(\alpha_1))q(\rho_\mathit{v}\circ\sigma_1)(u)+\cdots+q(\rho_\mathit{v}\circ\sigma_d)(\lambda)(\rho_\mathit{v}\circ\sigma_d)(u)\right|_{\mathit{v}}\\
&=\prod_{j=1}^d\prod_{\mathit{v}\in S_j}\left|-\rho_\mathit{v}\left(\frac{a_2}{a_1}\right)p_2-\cdots-\rho_\mathit{v}\left(\frac{a_r}{a_1}\right)p_r+c_{\mathit{v}(1)}q\sigma_{\mathit{v}(1)}(u)+\cdots+c_{\mathit{v}(d)}q\sigma_{\mathit{v}(d)}(u)\right|_{\mathit{v}}.
\end{align*}
Hence, by \eqref{eq2.19}, we have
\begin{equation}\label{eq2.20}
\prod_{j=1}^d\prod_{\mathit{v}\in S_j}\left|-\rho_\mathit{v}\left(\frac{a_2}{a_1}\right)p_2-\cdots-\rho_\mathit{v}\left(\frac{a_r}{a_1}\right)p_r+c_{\mathit{v}(1)}q\sigma_{\mathit{v}(1)}(u)+\cdots+c_{\mathit{v}(d)}q\sigma_{\mathit{v}(d)}(u)\right|_{\mathit{v}}<\frac{1}{H^\epsilon(u)q^{\frac{d}{r}+\epsilon}}.
\end{equation}
Since the inequality \eqref{eq2.20} holds  true for all the tuples $(u, q, p_1, \ldots, p_r) \in \mathcal{B}_1$,  we  can apply  Theorem \ref{schli} suitably. 
\vspace{.2cm}

For each $\mathit{v}\in M_\infty,$ since $c_{\mathit{v}(1)},\ldots,c_{\mathit{v}(d)}\in K$, not all zero, we let $k_v \in \{v(1), \ldots, v(d)\}$ such that $c_{k_v}\ne 0$.  
Now, for each $\mathit{v}\in S$, we define $r+d-1$ linearly independent linear forms in $r+d-1$  variables as follows:  for each $j = 1, 2, \ldots, d$ and for each  $\mathit{v}\in S_j$,  we define  
$$
L_{\mathit{v},i}(x_1, \ldots, x_r, \ldots, x_{r+d-1}) = x_{i} - \rho_{\mathit{v}}(\alpha_i) x_{r+j-1}
$$
for each $i = 1, 2, \ldots, r-1$; 
$$
L_{\mathit{v}, r-1+k_v}(x_1,\ldots,x_r,\ldots,x_{r+d-1})=-\rho_\mathit{v}\left(\frac{a_2}{a_1}\right)x_1-\cdots-\rho_\mathit{v}\left(\frac{a_r}{a_1}\right)x_{r-1}+c_{\mathit{v}(1)}x_{r-1+\mathit{v}(1)}+\cdots+c_{\mathit{v}(d)}x_{r-1+\mathit{v}(d)};
$$
 for each $r\leq m \neq r-1+k_v\leq r+d-1$, we define 
\begin{eqnarray*} 
L_{v, m}(x_1,\ldots,x_r,\ldots,x_{r+d-1})&=&x_m;
\end{eqnarray*}
and for each $\mathit{v}\in S\backslash M_\infty$ and  for each integer $i$ in the range $1\leq i \leq r+d-1$,  we consider 
$$
L_{v,i}(x_1,\ldots,x_r,\ldots,x_{r+d-1})=x_i.
$$
Since  $c_{k_v} \ne 0$, it follows that for each $\mathit{v}\in S$, the linear forms $L_{\mathit{v},1},\ldots,L_{\mathit{v}, r+d-1}$  are linearly  independent. 
\smallskip

Write the special points in $K^{r+d-1}$ as 
$$
\mathbf{X}=(p_2,\ldots,p_r,q\sigma_1(u),\ldots,q\sigma_d(u))\in  K^{r+d-1}.
$$
Then, by Theorem \ref{schli}, there exists a proper subspace of ${K}^{r+d-1}$ which contain infinitely many points $\mathbf{X}=(p_2,\ldots,p_r,q\sigma_1(u),\ldots,q\sigma_d(u))$.   Hence, we get a non-trivial relation  
\begin{equation}\label{eq2.21}
a'_2 p_2+\cdots+a'_r p_r+b'_1q\sigma_1(u)+\cdots+b'_d q \sigma_d(u)=0,\quad a'_i, b'_i\in K
\end{equation} 
holds true for infinitely many tuples $(p_2, \ldots, p_r, q\sigma_1(u), \ldots, q\sigma_d(u))$.  By Claim 1, we can always assume that not all $b_i$'s are zero. Thus,  we obtain a non-trivial relation in $r+d-1$ tuples. 

\bigskip

By continuing this process,  inductively, we can get a non-trivial relation with $a_i = 0$ for all integers $i = 1, 2, \ldots, r-1$. That is,  
$$
a''_r p_r+b''_1q\sigma_1(u)+\cdots+b''_d q \sigma_d(u)=0,\quad a''_r,~~ b''_i\in K
$$
holds true for infinitely many  tuples $(p_r, q\sigma_1(u), \ldots, q\sigma_d(u))$ where $b''_i$'s are not all zero. This proves Claim 2.  We then conclude exactly as in \cite[ Lemma 3]{corv} to complete the proof of  this Lemma.
\end{proof}

\begin{lemma}\label{lem3.2}
Let $K$ be a number field of degree $n$ which is  Galois  over $\mathbb{Q}$  and $S$  be a finite set of places on $K$ which contains all the archimedean places.  Let  $k\subset K$  be a subfield of degree $d$  over $\mathbb{Q}$ for some integer $d\geq 1$ and $\alpha_1, \ldots, \alpha_d$ be any  elements of $K$. For a given real number $\epsilon>0$,  let
\begin{equation}\label{eq2.22}
\mathcal{B} = \left\{(u,q, p) \in (\mathcal{O}_S^\times\cap k)\times\mathbb{Z}^2  :  0<|\alpha_1 q u_1+\cdots+\alpha_d q u_d-p|<\frac{1}{H^\varepsilon(u_1)|q|^{d+\varepsilon}}
\right\}
\end{equation}
where $u=u_1$ and  $u_2, \ldots, u_d$ are the other conjugates of $u$   and  for each triple $(u,q, p)\in \mathcal{B}$, there exists an integer $i\in\{1,\ldots,d\}$ such that  $|q\alpha_iu_i|>1$ and $q\alpha_iu_i$ is not a pseudo-Pisot number.  If $\mathcal{B}$ is infinite, then  there exist a proper subfield $k'\subset k$, a non-zero element $u'\in k$  and an infinite subset $\mathcal{B}'\subset \mathcal{B}$  such that for all triples $(u,q,p)\in\mathcal{B}'$, we have  $u/u'\in k'.$
\end{lemma}
\begin{proof} Note that $d\geq 2$ because $\mathbb{Q}$ doesn't admit any proper subfield in it.  Let $\mathcal{H}:=\mbox{Gal}(K/k)\subset\mbox{Gal}(K/\mathbb{Q})=\mathcal{G}$  be the subgroup of the Galois group $\mathcal{G}$ fixing $k$.  
Since $K$ is Galois over $\mathbb{Q}$, we have $K$ is Galois over $k$ and  $|\mathcal{G}/\mathcal{H}| =d$.  Therefore,  among the $n$ embeddings on $K,$ there are  $d$ embeddings $\sigma_1,\ldots,\sigma_d$ (with $\sigma_1$ is the identity) which are  the complete set of representatives of the left cosets of $\mathcal{H}$ in $\mathcal{G}$ and more precisely, we have 
$$
\mathcal{G}/\mathcal{H}:=\{\mathcal{H}, \sigma_2 \mathcal{H},\ldots,\sigma_d \mathcal{H}\}.
$$
For each $\rho\in\mbox{Gal}(K/\mathbb{Q})$ and for any triple $(u, q, p)\in\mathcal{B}$,  with the  rule  in  \eqref{newarchimedean}, we have 
\begin{align}\label{eq2.23}
|\alpha_1 q u_1+\cdots+\alpha_d q u_d-p|^{d(\rho)/[K:\mathbb{Q}]}&=|\rho(\alpha_1)\rho(q u_1)+\cdots+\rho(\alpha_d)\rho(q u_d)-\rho(p)|_\rho\nonumber\\
&=|\rho(\alpha_1)q\rho(u_1)+\cdots+\rho(\alpha_d)q\rho(u_d)-p|_\rho.
\end{align}
For each $\mathit{v}\in M_\infty$, let $\rho_\mathit{v}$  be an automorphism defining the valuation $\mathit{v}$, according to   \eqref{newarchimedean}: $|x|_\mathit{v}:=|x|_{\rho_\mathit{v}}$. Then the set $\{\rho_\mathit{v} : \mathit{v}\in M_\infty\}$  represents the left cosets of the subgroup generated by the complex conjugation in $\mathcal{G}$.   For $j=1,2,\ldots,d$, let $S_j$ be the subset of $M_\infty$  formed by those valuation $\mathit{v}$ such that $\rho_\mathit{v}|_k=\sigma_j:k\rightarrow \mathbb{C}$.  Note that $S_1\cup\ldots\cup S_d=M_\infty$. Thus, we have $M_\infty = \{\rho_v : v\in M_\infty\}$ and  
for each triple $(u, q, p) \in \mathcal{B}$, we obtain
\begin{eqnarray*}
\prod_{\mathit{v}\in M_\infty}|\rho_\mathit{v}(\alpha_1)q\rho_\mathit{v}(u_1)+\cdots&+&\rho_\mathit{v}(\alpha_d)q\rho_\mathit{v}(u_d)-p|_\mathit{v}\\ &=&\prod_{j=1}^d\prod_{\mathit{v}\in S_j}|\rho_\mathit{v}(\alpha_1)q \sigma_j(u_1)+\cdots+\rho_\mathit{v}(\alpha_d)q \sigma_j(u_d)-p|_\mathit{v}.
\end{eqnarray*} 
By \eqref{eq2.23}, we see that 
\begin{eqnarray*}
\prod_{\mathit{v}\in M_\infty}|\rho_\mathit{v}(\alpha_1)q\rho_\mathit{v}(u_1)+\cdots+\rho_\mathit{v}(\alpha_d)q\rho_\mathit{v}(u_d)-p|_\mathit{v} &=&\prod_{\mathit{v}\in M_\infty}|\alpha_1 q u_1+\cdots+\alpha_d q u_d-p|^{d(\rho)/[K:\mathbb{Q}]}\\
&=&|\alpha_1 q u_1+\cdots+\alpha_d q u_d-p|^{\sum_{\mathit{v}\in M_\infty}d(\rho_\mathit{v})/[K:\mathbb{Q}]}.
\end{eqnarray*}
Then, by the formula $\displaystyle\sum_{\mathit{v}\in M_\infty}d(\rho_\mathit{v})=[K:\mathbb{Q}]$, it follows that 
\begin{equation}\label{eq2.24}
\prod_{j=1}^d\prod_{\mathit{v}\in S_j}|\rho_\mathit{v}(\alpha_1)q \sigma_j(u_1)+\cdots+\rho_\mathit{v}(\alpha_d)q \sigma_j(u_d)-p|_\mathit{v}=|\alpha_1 q u_1+\cdots +\alpha_d q u_d-p|.
\end{equation}
 Now, for each $\mathit{v}\in S$, we define $d+1$ linearly independent linear forms in $d+1$  variables as follows:  For each $j = 1, 2, \ldots, d$ and  for an archimedean place $\mathit{v}\in S_j$, we define  
\begin{eqnarray*}
L_{\mathit{v},0}(x_0,\ldots,x_d)&=& -x_0 +\rho_\mathit{v}(\alpha_1) x_{1}+\cdots+\rho_\mathit{v}(\alpha_d) x_{d}
\end{eqnarray*}
 and for any integer $i$ satisfying $0< i\leq d$, we define 
$$
L_{v, i}(x_0,\ldots,x_d)=x_i.
$$
Also, for any $v\in S\backslash M_\infty$ and for any integer $i$ satisfying $0\leq i\leq d$, we put
$$
L_{v, i}(x_0, \ldots, x_d) = x_i.
$$
Clearly,  these linear forms are $\mathbb{Q}$-linearly independent. Let the special points $\mathbf{X} \in K^{d+1}$ be of the form
$$
\mathbf{X}=(p,q\sigma_1(u),\ldots,q\sigma_d(u)) \in K^{d+1}.
$$
In order to apply Theorem \ref{schli}, we need to estimate the following quantity 
\begin{equation}\label{eq2.25}
\prod_{\mathit{v\in S}}\prod_{j=0}^{d}\frac{|L_{\mathit{v},j}(\mathbf{X})|_\mathit{v}}{\Vert\mathbf{X}\Vert_\mathit{v}}.
\end{equation}
Using the fact that $L_{\mathit{v},j}(\mathbf{X}) =q \sigma_j(u)$, for all $1\leq j\leq d$ and for all $v$, we obtain
\begin{eqnarray*}
\prod_{\mathit{v\in S}}\prod_{j=1}^{d}|L_{\mathit{v},j}(\mathbf{X})|_\mathit{v}&=&\prod_{\mathit{v}\in S}\prod_{j=1}^{d}|q \sigma_j(u)|_\mathit{v}
= \prod_{\mathit{v}\in S}\prod_{j=1}^{d}|q|_\mathit{v}\prod_{j=1}^{d}\prod_{\mathit{v}\in S}|\sigma_j(u)|_\mathit{v}.
\end{eqnarray*}
Since $\sigma_j(u)$  are $S$-units, then by the product formula we obtain
$$
\prod_{\mathit{v}\in S}|\sigma_j(u)|_\mathit{v}=\prod_{\mathit{v}\in M_K}|\sigma_j(u)|_\mathit{v}=1.  
$$
Consequently, from the above equality, we get  
$$
\prod_{\mathit{v\in S}}\prod_{j=1}^{d}|L_{\mathit{v},j}(\mathbf{X})|_\mathit{v}=\prod_{\mathit{v}\in S}\prod_{j=1}^{d}|q|_\mathit{v}\leq \prod_{v\in M_\infty}\prod_{j=1}^{d}|q|_\mathit{v}=\prod_{j=1}^{d}|q|^{\sum_{\mathit{v}\in M_\infty}d(\rho_\mathit{v})/[K:\mathbb{Q}]}.
$$
Then, from the formula $\sum_{\mathit{v}\in M_\infty}d(\rho_\mathit{v})=[K:\mathbb{Q}]$, we get 
\begin{equation}\label{eq2.26}
\prod_{\mathit{v\in S}}\prod_{j=1}^{d}|L_{\mathit{v},j}(\mathbf{X})|_\mathit{v}\leq \prod_{v\in M_\infty}\prod_{j=1}^{d}|q|_\mathit{v}=\prod_{j=1}^{d}|q|^{\sum_{\mathit{v}\in M_\infty}d(\rho_\mathit{v})/[K:\mathbb{Q}]}=|q|^d.
\end{equation}
Now we estimate the product of the denominators in \eqref{eq2.25} as follows; Consider 
$$
\prod_{\mathit{v\in S}}\prod_{j=0}^{d}\Vert\mathbf{X}\Vert_\mathit{v}\geq \prod_{\mathit{v\in M_K}}\prod_{j=0}^{d}\Vert\mathbf{X}\Vert_\mathit{v}=\prod_{j=0}^{d}\left(\prod_{\mathit{v\in M_K}}\Vert\mathbf{X}\Vert_\mathit{v}\right),
$$
since $||\mathbf{X}||_\mathit{v}\leq 1$ for all $\mathit{v}\not\in S$. Thus, by the definition of $H(\mathbf{X})$, we conclude that 
\begin{equation}\label{eq2.27}
\prod_{\mathit{v\in S}}\prod_{j=0}^{d}\Vert\mathbf{X}\Vert_\mathit{v}\geq \prod_{j=0}^{d} H(\mathbf{X}).
\end{equation}
By \eqref{eq2.24}, \eqref{eq2.26}  and \eqref{eq2.27}, we get 
$$
\prod_{\mathit{v\in S}}\prod_{j=0}^{d}\frac{|L_{\mathit{v},j}(\mathbf{X})|_\mathit{v}}{||\mathbf{X}||_\mathit{v}}\leq \frac{1}{H^{d+1}(\mathbf{X})}|q|^d|\alpha_1 q u_1+\cdots+\alpha_d q u_d-p|.
$$
Therefore, by \eqref{eq2.22}, we get 
$$
\prod_{\mathit{v\in S}}\prod_{j=0}^{d}\frac{|L_{\mathit{v},j}(\mathbf{X})|_\mathit{v}}{||\mathbf{X}||_\mathit{v}}\leq \frac{1}{H^{d+1}(\mathbf{X})}|q|^d\frac{1}{H^{\epsilon}(u)}\frac{1}{|q|^{d+\varepsilon}}=\frac{1}{H^{d+1}(\mathbf{X})}\frac{1}{(|q|H(u))^{\varepsilon}}.
$$
First note that 
$$
|p| \leq |\alpha_1 q u_1+\cdots +\alpha_dqu_d|+1 \leq |q| |\alpha| \  H(u)^d \ d
$$
where $|\alpha| = \max\{|\alpha_i| : \ i = 1,2,\ldots, d\}$ and every conjugate of $u$ has absolute value bounded  by its Weil height power $d$.  Hence, we get
$$
|p|  \leq C(\alpha, d) |q| H^d(u),
$$
where $C(\alpha, d)$ is a positive constant depends only on $\alpha$ and $d$.  Since $H(\mathbf{X})\leq |q||p|H(u)^d$, we get    
$$
H(\mathbf{X}) \leq C'|q|^2H(u)^{2d} \leq C' (|q|H(u))^{2d}\Longrightarrow |q|H(u) \geq C'' H(\mathbf{X})^{1/(2d)}.
$$
where $C'$ and $C''$ are positive constants depends only on $\alpha$ and $d$  and hence  the last  inequality becomes 
$$
\prod_{\mathit{v\in S}}\prod_{j=0}^{d}\frac{|L_{\mathit{v},j}(\mathbf{X})|_\mathit{v}}{||\mathbf{X}||_\mathit{v}}\leq \frac{1}{H(\mathbf{X})^{d+1+\varepsilon'}},
$$
for some $\epsilon' >0$ which holds for infinitely many points $\mathbf{X}$.  Therefore, by Theorem \ref{schli},  there exists a proper subspace of $K^{d+1}$ which contain infinitely many points $\mathbf{X}=(p, q\sigma_1(u),\ldots,q\sigma_d(u))$.   It means that, we obtain a non-trivial  relation  
\begin{equation}\label{eq2.28}
a_0 p+a_1q\sigma_1(u)+\cdots+a_d q \sigma_d(u)=0,\quad a_i\in K,
\end{equation}
satisfied by all the triples $(u,q, p)\in\mathcal{B}_1 \subset \mathcal{B}$ for some infinite subset $\mathcal{B}_1$  of $\mathcal{B}$. Also,  for each triple $(u, q, p) \in \mathcal{B}_1$, without loss of generality, we can assume that $|q\alpha_1 u| > 1$ and $q\alpha_1 u$ is not a  pseudo-Pisot number. 

\bigskip

Since not all $a_i$'s are $0$, clearly, we can conclude that  at least one among $a_1,\ldots,a_d$  is non-zero.   Now, we have the following claim. 

\bigskip

\noindent{\bf Claim 1.}  There exists  a non-trivial relation as in \eqref{eq2.28} with $a_0 = 0$.

\bigskip
 
Suppose that $a_0\neq 0$.  Then we rewrite the relation \eqref{eq2.28} as 
\begin{equation}\label{eq2.29} 
p=-\frac{a_1}{a_0}q\sigma_1(u)-\cdots-\frac{a_d}{a_0}q\sigma_d(u).
\end{equation}
By considering the case when   $\sigma_j(a_1/a_0)\neq a_j/a_0,$ for some index $j\in\{2,\ldots,d\}$, or the case when $a_j/a_0=\sigma_j(a_1/a_0)$  for all $j$, we can conclude that  all the coefficients $a_j/a_0$  are non-zero.  Former case can be handled as in Case 1 in Claim 2 of Lemma \ref{lem3.1}. We deal with the latter case.

Put $\lambda=-a_1/a_0$.  With these notations,  we can re-write \eqref{eq2.29} as follows:
$$
p=q(\sigma_1(\lambda)\sigma_1(u)+\cdots+\sigma_d(\lambda)\sigma_d(u)).
$$
In the proof of Claim 2 of  Lemma \ref{lem3.1}, we had the two possibilities, namely,  either $\lambda\in k$ or $\lambda\notin k$.    Here also, the proof for the case $\lambda\notin k$ is similar to that of the proof of  Claim 2 of Lemma \ref{lem3.1}.  Therefore we consider the  case $\lambda\in k$.   By adding $-\alpha_1 q \sigma_1(u)-\cdots-\alpha_d q \sigma_d(u)$ to both sides in the above equality,  we get
\begin{align*}
|p-(\alpha_1 q \sigma_1(u)+\cdots+\alpha_d q \sigma_d(u))|&=|p-(\alpha_1 q u_1+\cdots+\alpha_d q u_d)|\\&=|(\lambda-\alpha_1)q \sigma_1(u)+(\sigma_2(\lambda)-\alpha_2)q \sigma_2(u)+\cdots+(\sigma_d(\lambda)-\alpha_d)q \sigma_d(u)|.
\end{align*}
Therefore by \eqref{eq2.22}, we get 
$$
0<|(\lambda-\alpha_1)q \sigma_1(u)+(\sigma_2(\lambda)-\alpha_2)q \sigma_2(u)+\cdots+(\sigma_d(\lambda)-\alpha_d)q \sigma_d(u)|<\frac{1}{|q|^{d+\varepsilon}}\frac{1}{H^\varepsilon(u)}.
$$
Then  dividing by $q$ on both sides to get
\begin{equation}\label{eq2.30}
0<|(\lambda-\alpha_1)\sigma_1(u)+(\sigma_2(\lambda)-\alpha_2)\sigma_2(u)+\cdots+(\sigma_d(\lambda)-\alpha_d)\sigma_d(u)|<\frac{1}{|q|^{d+1+\varepsilon}}\frac{1}{H^\varepsilon(u)}\leq \frac{1}{|q|^{1+\varepsilon}}\frac{1}{H^\varepsilon(u)}.
\end{equation}  
By putting  $\sigma_i(\lambda)-\alpha_i =\beta_i$ for all integers $i$,  we note that not all $\beta_i$'s are zero.  Then by  re-writing  \eqref{eq2.30}, we get   
\begin{equation}\label{eq2.31}
|\beta_1\sigma_1(u)+\cdots+\beta_d\sigma_d(u)|<\frac{1}{|q|^{1+\varepsilon}}\frac{1}{H^\varepsilon(u)}
\end{equation}
holds true for all triples $(u, q, p)\in \mathcal{B}_1$.  In order to apply Lemma \ref{lem2.1},  we distinguish two cases, namely, $\beta_1 =0$ and $\beta_1\neq 0$ as follows. 

\bigskip
  
Suppose $\beta_1=0$.  In this case, $\sigma_1(\lambda) = \alpha_1$ and hence the algebraic number $q\alpha_1 u=q\lambda u$.  Since $\alpha_1 q u$  is not a pseudo-Pisot number, we get  $q\lambda u$ is not a pseudo-Pisot number.  Therefore
$$
\max\{|\sigma_2(q \lambda u)|,\ldots,|\sigma_d(q \lambda u)|\}  \geq 1.
$$    
This implies that 
\begin{equation}\label{eq2.32}
\max\{|\sigma_2(u)|,\ldots,|\sigma_d(u)|\}\geq \frac{1}{|q|}\max \{|\sigma_2(\lambda)|,\ldots,|\sigma_d(\lambda)|\}^{-1}.
\end{equation} 
Since not all  $\beta_i$'s are  zero, let  $\beta_{i_1},\ldots,\beta_{i_r}$  be non-zero elements among $\beta_2,\ldots,\beta_d$ and  by \eqref{eq2.31}, we get 
\begin{equation}\label{eq2.33}
|\beta_{i_1}\sigma_{i_1}(u)+\cdots+\beta_{i_r}\sigma_{i_r}(u)|<\frac{1}{|q|^{1+\epsilon}}\frac{1}{H^\varepsilon(u)}.
\end{equation}
holds true for all triples  $(u,q,p)\in\mathcal{B}_1.$  Thus by \eqref{eq2.32}  and \eqref{eq2.33},  for all triples $(u, q,p)\in\mathcal{B}_1$,    the inequality
\begin{equation*}
|\beta_{i_1}\sigma_{i_1}(u)+\cdots+\beta_{i_r}\sigma_{i_r}(u)|<\mbox{max}\{|\sigma_1(u)|, \ldots, |\sigma_d(u)|\} H^{-\varepsilon}(u)
\end{equation*}
holds true.  Therefore by  Lemma \ref{lem2.1},  with the distinguished place $\omega$ corresponding to the identity embedding,  $n=i_r$ and $\lambda_{i_j}=\beta_{i_j}$ for $1\leq j\leq r$, we get an infinite subset $\mathcal{B}_2 \subset \mathcal{B}_1$ such that for all triples $(u,q, p) \in \mathcal{B}_2$ 
there exists a non-trivial relation of the form
$$
s_1q\sigma_1(u)+\cdots+s_d q \sigma_d(u)=0
$$
holds true for some $s_1, \ldots, s_d\in K$.  Therefore, by Lemma \ref{lem2.2}, there exist an infinite subset $\mathcal{B}_3$ of  $\mathcal{B}_2$  and a non-trivial relation of the form $a\sigma_j(u)+b\sigma_i(u)=0$   for some distinct integers $i$ and $j$    and $a,b\in K^\times$  satisfied by all the triples $(u,q, p)\in\mathcal{B}_3$.  
Hence, 
$$
-\sigma^{-1}_i\left(\frac{a}{b}\right)(\sigma^{-1}_i\circ \sigma_j)(u)=u$$
is true for all triples $(u, q, p) \in \mathcal{B}_3$.   Therefore, for any two triples $(u', q',p')$,  $(u'', q'', p'')$ $\in$ $\mathcal{B}_3$,  we have  
$$
\sigma^{-1}_i \circ\sigma_j(u'/u'')=u'/u''.
$$
That is,  the element $u'/u''$  is fixed by the automorphism $\sigma_i^{-1}\circ \sigma_j\notin \mathcal{H}$, and hence  $u'/u''$ belongs to the proper subfield 
$k'$ of $k$ which is fixed by the subgroup generated by $\mathcal{H}$ and $\sigma^{-1}_i\circ \sigma_j$.  To finish the proof of this  Lemma, fix a nonzero $u' \in k$ with $(u', q, p) \in \mathcal{B}_3$ and take any other triple  $(u,q, p) \in \mathcal{B}_3$, then we can get $u/u' \in k'$.

Now we assume that $\beta_1\ne 0$. In this case, the term $\beta_1 \sigma_1(u)$ does appear in \eqref{eq2.31}. Since   $|\alpha_1 qu_1|=|\alpha_1 q u|>1$, we see that  
$$
\max\{|u_1|,\ldots,|u_d|\}=\{|\sigma_1(u)|,\ldots,|\sigma_d(u)|\}\geq |u|>|\alpha_1|^{-1} |q|^{-1}
$$
 holds true for all pairs $(u, q)$ where the triples $(u,q,p)$  satisfying \eqref{eq2.31}. Thus  by \eqref{eq2.31}, we deduce that 
$$
0<|\beta_1\sigma_1(u)+\cdots+\beta_d\sigma_d(u)|<\frac{1}{|q|^{1+\varepsilon}}\frac{1}{H^\varepsilon(u)}<\frac{|\alpha_1|\max\{|\sigma_1(u)|,\ldots,|\sigma_d(u)|\}}{|q|^{\epsilon} H^\epsilon(u)}.
$$
By applying Lemma \ref{lem2.1} with the distinguished place $\omega$ as in the case $\beta_1=0$ and with the inputs $n=d$, $\lambda_i=\beta_i$ for each integer $i = 1, \ldots, d$, we conclude the same as in the case when $\beta_1=0$. This  completes the proof of the lemma.
\end{proof}

\section{Proof of Theorems \ref{maintheorem}, \ref{maintheorem1} and Corollary \ref{cor1}} 

\noindent{\bf Proof of Theorem \ref{maintheorem}.~} Since $\Gamma$ is finitely generated multiplicative subgroup of $\overline{\mathbb{Q}}^\times$, by enlarging $\Gamma$, if necessary, we can reduce to the situation where  $\Gamma\subset\overline{\mathbb{Q}}^\times$ is a group of $S$-units, namely,
$$
\Gamma = \mathcal{O}^\times_S=\{u\in K:\prod_{\mathit{v}\in S}|u|_\mathit{v}=1\}
$$
of a suitable  number field $K$ which is Galois over $\mathbb{Q}$,  with   $\alpha_1,\ldots,\alpha_r$  in $K$ and for a suitable finite set of places $S$ of $K$ which contains all the archimedean places. Also, note that $S$ is stable under Galois conjugation. 
 
\smallskip

Suppose  the assertion is not true. That is, the subset $\mathcal{B}$ (which is defined in \eqref{eq1.1}) is an infinite set.  Then by inductively,  we construct  sequences $\{\alpha^{(1)}_i\}_{i\geq 0},\ldots,\{\alpha^{(r)}_i\}_{i\geq 0}$, whose elements are in  $K$,  with the property that for any integer $n\geq 0$, the numbers $(\alpha^{(1)}_0\cdots\alpha^{(1)}_n)$, $\ldots$, $(\alpha^{(r)}_0\cdots\alpha^{(r)}_n)$ are $\mathbb{Q}$-linearly independent, an infinite decreasing chain $\mathcal{B}_i$ of an infinite subset of $\mathcal{B}$  and an infinite strictly decreasing chain $k_i$ of subfields of $K$ satisfying the following:
\bigskip

{\it For each integer $n\geq 0$, $\mathcal{B}_n\subset (k_n\times\mathbb{Z}^{r+1})\cap\mathcal{B}_{n-1}$, $k_n\subset k_{n-1}$, $k_n\neq k_{n-1}$  and  all but finitely many tuples $(u, q,p_1,\ldots,p_r)\in\mathcal{B}_n$   satisfying the inequalities:      $|\alpha^{(i)}_0\cdots\alpha^{(i)}_n qu|>1$,  $\alpha^{(i)}_0\cdots\alpha^{(i)}_n qu$ is not a pseudo-Pisot number  for some integer $i\in \{1,\ldots, r\}$, and  
\begin{equation}\label{eq4.1}
|\alpha^{(j)}_0\cdots\alpha^{(j)}_n q u-p_j|<\frac{1}{H^{\epsilon/(n+1)}(u)|q|^{\frac{d}{r}+\varepsilon}} \mbox{ for each integer }  j = 1, 2, \ldots, r. 
\end{equation}}
If such sequences exist, then  we  eventually get a contradiction to  the fact that the number field $K$ does not admit an infinite  strictly  decreasing chain of subfields.  Therefore in order to finish the proof of the theorem, it  suffices to construct such sequences.
\smallskip

We proceed our construction by applying induction on $n$: for $n=0$, put $\alpha^{(j)}_0=\alpha_j$ for $1\leq j\leq r$, $k_0=K$  and $\mathcal{B}_0=\mathcal{B}$, and we are done in this case because of our supposition. 

 By the induction hypothesis, we assume that  $\alpha^{(j)}_n$, $k_n$  and $\mathcal{B}_n$ for an integer $n\geq 0$ exist with the property that $(\alpha^{(1)}_0\cdots\alpha^{(1)}_n),\ldots,(\alpha^{(r)}_0\cdots\alpha^{(r)}_n)$ are $\mathbb{Q}$-linearly independent  and  satisfying \eqref{eq4.1}. Now we prove $n+1$-th stage. 
  
For each integer $j  = 1, 2, \ldots, r$, we let 
$$
\delta_j=\alpha^{(j)}_0\cdots\alpha^{(j)}_n.
$$
By the induction hypothesis,  the numbers $\delta_1,\ldots,\delta_r$  are $\mathbb{Q}$-linearly independent  and satisfies \eqref{eq4.1}.  Then by applying Lemma \ref{lem3.1} with $\delta_1,\ldots,\delta_r$, $k=k_n$, we obtain  an element $\gamma_{n+1}\in k_n$, a proper subfield $k_{n+1}$ of $k_n$  and an infinite set $\mathcal{B}_{n+1}\subset\mathcal{B}_n$  such that all tuples $(u, q, p_1,\ldots,p_r)\in\mathcal{B}_{n+1}$  satisfy $u=\gamma_{n+1}v$  with $v\in k_{n+1}$ and $\gamma_{n+1} \in k_n$. Note that since $u\in \mathcal{O}_S^\times$, we observe that $v\in \mathcal{O}_S^\times$. Hence, as $u$ varies, we see that $v$ also varies over $\mathcal{O}_S^\times.$ Thus, we can assume that $(u, q, p_1, \ldots, p_r) \in \mathcal{B}_{n+1}$ if and only if $(v, q, p_1, \ldots, p_r) \in \mathcal{B}_{n+1}$. 

\bigskip

 Set  $\alpha^{(j)}_{n+1}=\gamma_{n+1}$  for all $1\leq j\leq r$. Clearly, 
$$
\alpha_0^{(j)}\cdots\alpha_n^{(j)}\alpha^{(j)}_{n+1}=\delta_j \gamma_{n+1} := \delta_{n+1}^{(j)} \mbox{ for all } 1\leq j\leq r. 
$$
Therefore, by induction hypothesis, it is clear that $\delta_{n+1}^{(1)}, \ldots, \delta_{n+1}^{(r)}$ are 
 $\mathbb{Q}$-linearly independent. Also, by induction hypothesis, we know that  for every tuple $(u,q, p_1, \ldots, p_r) \in \mathcal{B}_{n+1}$, there exists an integer $i \in \{1, \ldots, r\}$ satisfying $|\delta_i qu| > 1$ and $\delta_iqu$ is not a pseudo-Pisot number. Since $\delta_j qu = \delta_jq\gamma_{n+1}v = \delta_{n+1}^{(j)}qv$, for every tuple 
 $(v, q, p_1, \ldots, p_r) \in \mathcal{B}_{n+1}$, there exists an integer  $i$ such that $|\delta_{n+1}^{(i)}qv| > 1$ and $\delta_{n+1}^{(i)}qv$ is not a pseudo-Pisot number and 
 $$
|\delta_{n+1}^{(j)}qv-p_j| = |\delta_j\gamma_{n+1}qv - p_j| = |\delta_jqu-p_j|<\frac{1}{H(\gamma_{n+1}v)^{\varepsilon/(n+1)}|q|^{\frac{d}{r}+\varepsilon}}.
$$
Since  $v\in k_{n+1}$, we see that 
 $$H(\gamma_{n+1}v)\geq H(\gamma_{n+1})^{-1} H(v),
 $$
 and hence in particular,  for almost
 all $v\in k_{n+1}$, we get $H(\delta_{n+1}v)\geq H(v)^{(n+1)/(n+2)}$. Therefore,  for all but finitely many tuples $(v, q, p_1, \ldots, p_r) \in \mathcal{B}_{n+1}$ and for all $1\leq j\leq r$, we have the following  inequality  
 $$
|\delta_{n+1}^{(j)}q v-p_j|<\frac{1}{H(v)^{\varepsilon/(n+2)}|q|^{\frac{d}{r}+\varepsilon}}.
$$
This proves the induction and hence the theorem. $\hfill\Box$

\bigskip

\noindent{\bf Proof of Theorem \ref{maintheorem1}.~} The proof of this theorem is similar to the proof of Theorem \ref{maintheorem}. 

Suppose that there are infinitely many triples $(u, q, p)\in \mathcal{O}^\times_S\times \mathbb{Z}^2$ satisfying the following inequality
$$
0<|\alpha_1 q \sigma_1(u)+\cdots+\alpha_d q\sigma_d(u)-p|\leq |q|^{-d-\varepsilon}H^{-\varepsilon}(u),
$$
where $\sigma_i$'s are all the embeddings of $\mathbb{Q}(u)$ to $\mathbb{C}$.  Then by inductively,  we construct sequences $\{\alpha_{i,0}\}_{i=0}^\infty$,$\ldots$, $\{\alpha_{i,d}\}_{i=0}^\infty$ whose elements are in  $K$, an infinite decreasing chain $\mathcal{B}_i$ of an infinite subset of $\mathcal{B}$  and an infinite strictly decreasing chain $k_i$ of subfields of $K$ satisfying the following properties:
\bigskip

{\it For each integer $n\geq 0$, $\mathcal{B}_n\subset (k_n\times \mathbb{Z}^2)\cap\mathcal{B}_{n-1}$, $k_n\subset k_{n-1}$, $k_n\neq k_{n-1}$  and all but finitely many triples $(u,q,p)\in\mathcal{B}_n$   satisfying the inequalities:      $|\alpha_{i,0}\cdots\alpha_{i,n} qu_i|>1$,  $\alpha_{i,0}\cdots\alpha_{i,n} qu_i$ is not a  pseudo-Pisot number for some integer $i\in\{1, \ldots, d\}$ and  
\begin{equation}\label{eq4.2}
|\alpha_{1,0}\cdots\alpha_{1,n}q\sigma_1(u)+\cdots+\alpha_{d,0}\cdots\alpha_{d,n}q\sigma_d(u)-p|<\frac{1}{H^{\varepsilon/(n+1)}(u)|q|^{d+\varepsilon}}.
\end{equation}}
If such sequences exist, then  we eventually get a contradiction to the fact that the number field $K$ does not admit any infinite  strictly  decreasing chain of subfields.  Therefore in order to finish the proof of the theorem, it  suffices to construct such sequences.
\smallskip

We proceed our construction by applying induction on $n$: for $n=0$, put $\alpha_{i,0}=\alpha_i$ for each integer $i= 1,\ldots, d$, $k_0=K$  and $\mathcal{B}_0=\mathcal{B}$, and we are done in this case because of our supposition. 
By the induction hypothesis, we assume that  $\alpha_{i,n}$, $k_n$  and $\mathcal{B}_n$ for an integer $n\geq 0$  and  \eqref{eq4.2} holds true. Then by applying Lemma \ref{lem3.2} with $k=k_n$  and 
$$
\delta_1=\alpha_{1,0}\cdots\alpha_{1,n},  \ldots, \delta_d = \alpha_{d,0}\cdots\alpha_{d,n},
$$
we obtain an element $\gamma_{n+1}\in k_n$, a proper subfield $k_{n+1}$ of $k_n$  and an infinite set $\mathcal{B}_{n+1}\subset\mathcal{B}_n$  such that all triples $(u,q,p)\in \mathcal{B}_{n+1}$  satisfy $u=\gamma_{n+1}v$  with $v\in k_{n+1}$. Note that since $u\in\mathcal{O}^\times_S$, we observe that $v\in\mathcal{O}^\times_S$.  Hence, as $u$ varies, we see that $v$ varies over $\mathcal{O}^\times_S$.  Thus,  we can assume that $(u, q, p)\in\mathcal{B}_{n+1}$ if and only if $(v, q, p)\in\mathcal{B}_{n+1}$.  
\bigskip

Set  $\alpha_{j,n+1}=\sigma_j(\gamma_{n+1})$  for all $1\leq j\leq d$.  Clearly,
$$
\alpha_{j,0}\cdots\alpha_{j,n}\alpha_{j,n+1}=\delta_j\sigma_j(\gamma_{n+1}):=\delta^{(j)}_{n+1}\mbox{~~for all~~} 1\leq j\leq d.
$$
By the induction hypothesis, we know that for every triple $(u,q,p)\in \mathcal{B}_{n+1}$, there exists an integer $i$ satisfying $|\delta_i qu|>1$  and $\delta_i qu$ is not a pseudo-Pisot number.  Since $\delta_j q u_j=\delta_j q \sigma_j(\gamma_{n+1}v)=\delta^{(j)}_{n+1}q\sigma_j(v),$  for every triple $(u,q,p)\in\mathcal{B}_{n+1}$, there exists an integer $i$  such that $|\delta^{(i)}_{n+1}q u_i|>1$  and $\delta^{(i)}_{n+1}qu_i$ is not a pseudo-Pisot number and 
\begin{align*}
|\delta^{(1)}_{n+1} q\sigma_1(v)+\cdots+\delta^{(i)}_{n+1} q\sigma_i(v)+\cdots+\delta^{(d)}_{n+1} q\sigma_d(v)-p|&=|\delta_1 q \sigma_1(\gamma_{n+1} v)+\cdots+\delta_d q \sigma_d(\gamma_{n+1} v)-p|\\&=|\delta_1 q u_1+\cdots+\delta_d q u_d-p|<\frac{1}{H^{\varepsilon/(n+1)}(\gamma_{n+1}v)|q|^{d+\varepsilon}}.
\end{align*}
Since $v\in K$, we see that 
$$ 
H(\gamma_{n+1}v)\geq H(\gamma_{n+1})^{-1} H(v),
$$ 
and hence in particular, for almost all $v\in K$,  we have $H(\delta_{n+1}v)\geq H(v)^{(n+1)/(n+2)}$. Therefore, for all but  finitely many such triple $(v,q,p)\in\mathcal{B}_{n+1}$, we have the following inequality 
$$
|\delta^{(1)}_{n+1} q v+\cdots+\delta^{(i)}_{n+1} q\sigma_i(v)+\cdots+\delta^{(d)}_{n+1} q\sigma_d(v)-p|<\frac{1}{H^{\varepsilon/(n+2)}(v)|q|^{d+\varepsilon}}
$$
holds true.  This proves the induction step  and hence the theorem.  $\hfill\Box$
\bigskip

\noindent{\bf Proof of Corollary \ref{cor1}.}  Suppose the assertion of Corollary \ref{cor1} is false. Then   $\alpha_1,\ldots,\alpha_r$  are algebraic numbers. By choosing $\varepsilon<\eta\log |\alpha|/\log H(\alpha)$, we see that there are  infinitely many tuples $(n,q,p_1,\ldots,p_r)$ in $\mathbb{Z}_{>0}^2\times \mathbb{Z}^r$   satisfying 
$$
0<|\alpha_j \alpha^n q -p_j|<\frac{1}{H(\alpha^n)^{\varepsilon}q^{\frac{d}{r}+\varepsilon}}\quad\mbox{for all~~} 1\leq j\leq r.
$$
Then by taking $\Gamma = \langle \alpha\rangle$  and $u = \alpha^n$ in  Theorem \ref{maintheorem},  we get,  for infinitely many values of $n$, $\alpha_i q \alpha^n$ is a pseudo-Pisot number for all $1\leq i\leq r$, and in particular all their other conjugates have modulus less than $1$.

\smallskip  

Let $K$ be the  Galois closure   of the number field $\Q(\alpha,\alpha_1,\ldots,\alpha_r)$ over $\Q$.  By our  assumption on $\alpha$, we know that $\alpha$ has a conjugate $\beta$ with $|\beta|>|\alpha|$. Therefore there exists an automorphism $\sigma:K\to K$ maps $\alpha$ to $\beta$. Hence,  for all $n\in\N$,  we have $\sigma(q\alpha_i \alpha^n)=q\sigma(\alpha_i)\beta^n$. Since $\alpha_i q \alpha^n$ is a pseudo-Pisot number for infinitely many values of $n$, we see that all the other conjugates of $\alpha_iq\alpha^n$ have modulus $<1$. In particular, the same is true for $\sigma(q\alpha_i \alpha^n)$. But, since $|\sigma(q\alpha_i\alpha^n)| = |q\sigma(\alpha_i)| |\beta|^n$, and $ |\beta| > |\alpha|> 1$, this is impossible.  This proves the corollary. $\hfill\Box$
\bigskip

\noindent{\bf Acknowledgements.} We are grateful to the referee for many constructive suggestions and pointing out many typos to make the article readable and clear. Also, we are thankful to NISER, Bhubaneswar where,  in December, 2019, this work was blossomed.    The first author is thankful to Professor P. Philippon for his useful discussion in the preliminary version of this article and encouragements in this work. The second author is thankful to SERB-METRICS project.


\bigskip

\begin{thebibliography}{9999}
\bibitem{bomb}
Bombieri, E., Gubler, W.: Heights in Diophantine geometry, New Mathematical Monographs, Vol. 4, Cambridge University Press, Cambridge (2006).
\bibitem{cH}
Corvaja, P., Han\v{c}l, J.: A transcendence criterion for infinite products, Atti Accad. Naz. Lincei Rend. Lincei Math. Appl. {\bf 18} (3), 295-303, (2007). 
\bibitem{corv} 
Corvaja, P., Zannier, U.: On the rational approximation to the powers of an algebraic number: Solution of two problems of Mahler and Mendes France, Acta Math. {\bf 193}, 175-191 (2004).
\bibitem{hancl}
Han\v{c}l,J.,  Kolouch, O., Pulcerov\'a, S., \v{S}t\v{e}pni\v{c}ka, J.:  A note on the transcendence of infinite products, Czech. Math. J., {\bf 62}(3), 613-623 (2012).
\bibitem{kul}
Kulkarni, A., Mavraki, N. M., Nguyen, K. D.:  Algebraic approximations to linear combinations of powers: An extension of results by Mahler and Corvaja-Zannier, Trans. Amer. Math. Soc., {\bf 371}, 3787-3804 (2019).
\bibitem{ridd1}
Ridout, D,: The $p$-adic generalization of the Thue-Siegel-Roth theorem, Mathematika, {\bf 5}, 122-124 (1958). 
\bibitem{roth1}
Roth, K.: Rational approximations to algebraic numbers,  Mathematika, {\bf 2 } (1), 1-20 (1955).
\bibitem{satha}
Saradha, N, Thangadurai, R.: Pillars of Transcendental Number Theory, Springer, Singapore (2020). 
\bibitem{schmidt}
Schmidt, W. M.:  Diophantine  Approximations and Diophantine Equations,  Lecture Notes in Math., {\bf 1467}, Springer, Berlin, 1991.
\bibitem{zannier}
Zannier, U.: Some Applications of Diophantine Approximation to Diophantine Equations (with special emphasis on the Schmidt Subspace Theorem), Forum, Udine (2003).
\end{thebibliography}
\end{document}